\documentclass[a4paper,12pt]{amsart}
\usepackage[top=1.5in, bottom=1in, left=0.9in, right=0.9in]{geometry}

\usepackage[utf8]{inputenc}
\usepackage[all]{xy}
\usepackage{color}
\usepackage{hyperref}

\usepackage{amsmath, mathtools}
\usepackage{hyperref}
\hypersetup{colorlinks=true}
\usepackage{amssymb}
\usepackage{amsfonts}
\usepackage{graphicx}
\usepackage{mathrsfs}
\usepackage{amscd}
\usepackage[all]{xy}
\usepackage{hyperref}
\usepackage{amsthm}
\usepackage{cleveref}
\usepackage{verbatim}
\usepackage{amscd}
\usepackage{mathtools}
\usepackage{comment}

\newtheorem{thm}{\bf Theorem}[section]

\newtheorem*{thma}{Theorem A}
\newtheorem*{thmb}{Theorem B}
\newtheorem*{thmc}{Theorem C}
\newtheorem*{thmd}{Theorem D}
\newtheorem*{thme}{Theorem E}

\newtheorem{prop}[thm]{\bf Proposition}
\newtheorem{lemma}[thm]{\bf Lemma}
\newtheorem{cor}[thm]{\bf Corollary}
\theoremstyle{definition}

\theoremstyle{remark}
\newtheorem{remark}[thm]{\bf Remark}
\newtheorem{Notation}[thm]{\bf Notation}
\newtheorem{Question}[thm]{\bf Question}
\newtheorem{question}[thm]{\bf Question}
\newtheorem{example}[thm]{\bf Example}
\numberwithin{equation}{section}
\newtheorem*{acknowledgement}{Acknowledgments}

\newcommand{\HH}[3]{\operatorname{H}^{#1}_{#2}(#3)}

\newcommand{\ol}[1]{\overline{#1}}

\DeclareMathOperator{\height}{{ht}}

\DeclareMathOperator{\depth}{{depth}}

\DeclareMathOperator{\Spec}{{Spec}}
\DeclareMathOperator{\Supp}{{Supp}}
\DeclareMathOperator{\Ass}{{Ass}}

\DeclareMathOperator{\Ext}{{Ext}}

\DeclareMathOperator{\conv}{{{conv}}}

\DeclareMathOperator{\Sing}{{{Sing}}}
\DeclareMathOperator{\codim}{{{codim}}}

\DeclareMathOperator{\st}{{star}}
\DeclareMathOperator{\adj}{{adj}}
\DeclareMathOperator{\supp}{{supp}}
\DeclareMathOperator{\ini}{{in}}
\DeclareMathOperator{\gin}{{gin}}
\DeclareMathOperator{\Proj}{{Proj}}
\DeclareMathOperator{\indeg}{{indeg}}
\DeclareMathOperator{\topdeg}{{topdeg}}
\DeclareMathOperator{\reg}{{reg}}
\DeclareMathOperator{\Sym}{{Sym}}
\DeclareMathOperator{\vol}{{vol}}
\newcommand{\ep}[1]{{\epsilon_{+}^{#1}}}
\newcommand{\emi}[1]{{\epsilon_{-}^{#1}}}
\newcommand{\eps}[1]{{\epsilon^{#1}}}

\newcommand{\Mod}[1]{\ (\mathrm{mod}\ #1)}

\def\ls{\leqslant}
\def\gs{\geqslant}

\def\ll{\lambda}

\def\fu{\mathbf{u}}
\def\f0{\mathbf{0}}
\def\fe{\mathbf{e}}

\def\fv{\mathbf{v}}
\def\fx{\mathbf{x}}
\def\fb{\mathbf{b}}

\def\fp{\mathbf{p}}

\def\fa{\mathbf{a}}
\def\fe{\mathbf{e}}

\def\fm{\mathfrak{m}}

\def\fp{\mathfrak{p}}

\def \PP{\mathbb P}
\def \QQ{\mathbb Q}
\def \CC{\mathbb C}
\def \RR{\mathbb R}

\def \FF{\mathbb F}
\def \ZZ{\mathbb Z}
\def \NN{\mathbb N}

\def \C{\mathcal C}

\def \O{\mathcal O}
\def \G{\mathcal G}
\def \R{\mathcal R}
\def \T{\mathcal T}

\def \P{\mathcal P}

\def \O{\mathcal O}
\def \F{\mathcal F}

\dedicatory{{Dedicated to Professor Gennady Lyubeznik on the~occasion of his~sixtieth~birthday.}}

\begin{document}

\title[Length of local cohomology of powers of ideals]{Length of local cohomology of powers of ideals}

\author[Hailong Dao]{Hailong Dao}
\address{Hailong Dao\\ Department of Mathematics \\ University of Kansas\\405 Snow Hall, 1460 Jayhawk Blvd.\\ Lawrence, KS 66045}
\email{hdao@ku.edu}

\author[Jonathan Monta\~no]{Jonathan Monta\~no}
\address{Jonathan Monta\~no \\ Department of Mathematics \\ University of Kansas\\405 Snow Hall, 1460 Jayhawk Blvd.\\ Lawrence, KS 66045}
\email{jmontano@ku.edu}

  \begin{abstract}
Let $R$ be a polynomial ring over a field $k$ with irrelevant ideal $\fm$ and dimension $d$. Let $I$ be a homogeneous ideal in $R$. We study the asymptotic behavior of the length of the modules $\HH{i}{\fm}{R/I^n}$ for $n\gg 0$. We show that for a fixed number $\alpha \in \ZZ$, $\limsup_{n\rightarrow \infty}\frac{\lambda(\HH{i}{\fm}{R/I^n}_{\gs -\alpha n}) }{n^d}<\infty.$ It follows that $\limsup_{n\rightarrow \infty}\frac{\lambda(\HH{i}{\fm}{R/I^n}) }{n^d}<\infty$ when $X = \Proj R/I$ is  locally a complete intersection (lci) and  $i\ls \dim X$. We also establish that the actual limit exists and is rational for certain classes of monomial ideals $I$ such  that the lengths of local cohomology of $I^n$ are eventually finite.  Our proofs use  Gr\"obner deformation and Presburger arithmetic. Finally, we utilize more traditional commutative algebra techniques to  show that $\liminf_{n\rightarrow \infty}\frac{\lambda(\HH{i}{\fm}{R/I^n})}{n^d}>0$ under certain conditions when $R/I$ is either $F$-pure or lci.
  \end{abstract}


\maketitle

\section{Introduction}

Let $(R,\fm, k)$ be a Noetherian local ring or standard graded ring (in which case $\fm$ is the irrelevant ideal), and set $d=\dim(R)$. Let $I$ be an ideal of $R$. This paper is motivated by some recent striking results on the local cohomology modules $\HH{i}{\fm}{R/I^n}$ of powers of $I$. 

For example, Cutkosky showed in \cite{CutAdv} that the limit $$ \lim_{n\rightarrow \infty}\frac{\lambda(\HH{0}{\fm}{R/I^n})}{n^d}$$
always exists and is finite when $R$ is analytically unramified or $d=0$. Here $\lambda(-)$ denotes  the length of a module. This limit was defined and studied as a $\limsup$ by Katz-Validashti in \cite{KV}, where it was  shown to have intimate connection to the analytic spread of $I$. Furthermore, Herzog, Puthenpurakal, and Verma, proved that the limit is rational when $I$ is a monomial ideal in a polynomial ring (\cite{HPV}). 

In a different vein, when $R = k[x_1,...,x_d]$ and $I$ is a homogeneous ideal, works by Bhatt-Blickle-Lyubeznik-Singh-Zhang and Raicu (\cite{BBLSZ, Claudiu}) revealed very surprising patterns on the graded pieces of $\HH{i}{\fm}{R/I^n}$ for  $i\gs 0$. For example, when $\Proj R/I$ has certain types of singularities and for $i$ in certain range, their results assert that the local cohomology modules satisfy ``Kodaira vanishing", namely $\HH{i}{\fm}{R/I^n}_{< 0} =0$ for all $n\gs 0$.  This type of vanishing will be crucial for some applications in our present work. 

Apart from the papers mentioned above, there have been numerous other interesting works on local cohomology of powers of an ideal. We refer to section \ref{openq} for a more detailed survey. The purpose of this paper is to investigate the asymptotic behavior of length of higher local cohomology modules of powers of $I$. In particular, we are interested in the following questions.

\begin{Question}\label{questions}
Fix an integer $i>0$.  
\begin{enumerate}
\item When is $\lambda(\HH{i}{\fm}{R/I^n})<\infty$ for $n\gg 0$?
\item Assume $\lambda(\HH{i}{\fm}{R/I^n})<\infty$ for $n\gg 0$, is  $\limsup_{n\rightarrow \infty}\frac{\lambda(\HH{i}{\fm}{R/I^n})}{n^d}<\infty$?
\item When is $\liminf_{n\rightarrow \infty}\frac{\lambda(\HH{i}{\fm}{R/I^n})}{n^d}>0?$
\item When does the limit $\lim_{n\rightarrow \infty}\frac{\lambda(\HH{i}{\fm}{R/I^n})}{n^d}$ exist?
\end{enumerate} 
\end{Question}

We are able to answer the questions above in many cases of interests. In the following we shall describe both the structure of the paper and its main results. For notational convenience let us define $$\ep{i}(I) := \limsup_{n\rightarrow \infty}\frac{\lambda(\HH{i}{\fm}{R/I^n})}{n^d},$$ $$\emi{i}(I) := \liminf_{n\rightarrow \infty}\frac{\lambda(\HH{i}{\fm}{R/I^n})}{n^d},$$ and 
$$\eps{i}(I) := \lim_{n\rightarrow \infty}\frac{\lambda(\HH{i}{\fm}{R/I^n})}{n^d}$$

Section \ref{prelim} establishes some preliminary results in quite general setting when $R$ is a local ring. Here we provide, via Grothendieck Finiteness Theorem, precise conditions for $\lambda(\HH{i}{\fm}{R/I^n})$ to be finite for $n\gg 0$. One notable result is Proposition \ref{H}, where we show a connection between  depth conditions on the Rees algebra of $I$ and the positivity of  $\ep{i}(I)$. 

For the rest of the paper we focus mostly on the case when $R=k[x_1,\ldots, x_d]$, $\fm=(x_1,\ldots, x_d)$, and $I$ is a homogeneous ideal. First, in Section \ref{smono1} we assume that $I$ is a monomial ideal. This case is very helpful for later applications and is interesting in its own right due to the intricate combinatorics involved. 

\begin{thma}$($Theorem \ref{mainMono}$)$
Let $I$ be a monomial ideal.  Assume $\lambda(\HH{i}{\fm}{R/I^n})<\infty$ for $n\gg 0$. Then the sequence $\{\lambda(\HH{i}{\fm}{R/I^n})\}_{n\gs 0}$ agrees with a quasi-polynomial of degree at most $d$ for $n\gg 0$. Moreover, this sequence has a rational generating function. 
\end{thma}

Recall that a {\it quasi-polynomial} on $n$ is a function $f:\NN\rightarrow \QQ$ for which there exist $\pi \in \NN$ and polynomials $p_i(n)\in \QQ[n]$ for $i=0,\ldots, \pi-1$ such that $f(n)=p_i(n)$ whenever $n\equiv i\, (\text{mod } \pi)$.

In fact, we are able to prove a  much stronger result: the rational generating function condition holds for any Noetherian graded family of ideals $(I_n)_{n\gs 0}$, not just ordinary powers. Interestingly, our proofs rest on the theory of Presburger arithmetic and counting functions, together with a formula by Takayama for local cohomology of monomial ideals. 

Next, we investigate situations where one can actually show that $\eps{i}(I)$ exists. This is the content of Section \ref{smono2}. The main result here is that the limit exists if we replace the regular powers $(I^n)_{n\gs 0}$ of a monomial ideal by their integral closures $(\ol{I^n})_{n\gs 0}$. 

\begin{thmb}$($Theorem \ref{mainNormal}$)$
 Let $I$ be a monomial ideal.  Assume $\lambda(\HH{i}{\fm}{R/\ol{I^n}})<\infty$ for $n\gg 0$, then the limit
$$\lim_{n\rightarrow \infty}\frac{\lambda(\HH{i}{\fm}{R/\ol{I^n}})}{n^d}$$
exists and is a rational number.
\end{thmb}

This implies, for instance, that $\eps{i}(I)$ exists and is rational when $I_{j}$ is normal for each $j$, where $I_{j}$ is the ideal obtained by setting $x_j=1$. We also analyze carefully the case of edge ideals in the rest of Section \ref{smono2}. 

Next we switch our attention to arbitrary homogeneous ideals. In Section \ref{lsup} we prove that as long as we restrict to a certain range, the $\limsup$ is finite. The proof uses Gr\"obner deformation to reduce to the monomial case. 

 \begin{thmc}$($Theorem \ref{gradedFiniteLimpup}$)$
 Let $R=k[x_1,\ldots,x_d]$ and $I$ be a homogeneous $R$-ideal,  then for every  $\alpha \in \ZZ$, we have $$\limsup_{n\rightarrow \infty}\frac{\lambda(\HH{i}{\fm}{R/I^n}_{\gs -\alpha n}) }{n^d}<\infty.$$
 \end{thmc}
 
As a corollary, when there is an $\alpha  \in \ZZ$ such that $\HH{i}{\fm}{R/I^n}_{<-\alpha n} =0$ for all $n\gg0$, then $\ep{i}(I)<\infty$. This is known when $X = \Proj R/I$ is  locally a complete intersection (lci) and $i\ls \dim X$ (cf. \cite[1.4]{BBLSZ} and Proposition \ref{Rob}), or when $I$ is a $GL$-invariant ideal that is the thickening of a determinantal ideal (cf. \cite[6.1]{Claudiu}). For precise statements, see  Corollaries \ref{lcilimsup} and \ref{determinantal}. 
 
In the next part of the paper we study the $\liminf$ denoted by  $\emi{i}(I)$ as above. Here we are able show that when $R/I$ is $F$-pure or lci, $\emi{i}(I)$ is often positive.  We start in the positive characteristic situation. 
  
 \begin{thmd} $($Theorem \ref{FpureLiminf}$)$ Let $(R,\fm)$ be a regular local ring of characteristic $p>0$ and dimension $d$. If $I$ is an $R$-ideal such that $R/I$ is $F$-pure, then for every $i$ such that $\HH{i}{\fm}{R/I}\neq 0$ we have $\emi{i}(I)>0.$

\end{thmd}

In fact, we prove a stronger statement, not just for $F$-pure ideals but for the larger class of $F$-full ideals. See Theorem \ref{FpureLiminf} for details.  We can also prove a similar result in the graded characteristic $0$ case. 

\begin{thme} $($Theorem \ref{SmoothLiminf}$)$
Let $R=k[x_1,\ldots,x_d]$ be a polynomial ring such that  $k$ is of characteristic 0. If $I$ is a homogeneous prime ideal such that $X=\Proj R/I$ is lci, then for every  $ i< \codim(\Sing X )$ such that $\HH{d-i}{I}{R}\neq 0$ we have $\emi{i}(I)>0.$
\end{thme}
\vspace{3mm}

In the last Section \ref{openq}, we discuss various works in the literature on related topics: $\varepsilon$-multiplicity, graded families of ideals, regularity of powers, and volume of convex bodies.  We also list some intriguing open questions.

\begin{acknowledgement}
We are grateful to Bhargav Bhatt, Dale Cutkosky, Linquan Ma, Anurag Singh, Kevin Woods, and Wenliang Zhang for many very helpful discussions. We are especially grateful to Robert Lazarsfeld for providing us with the proof of Proposition \ref{Rob}. We also wish to thank the referee for her or his helpful corrections. As with most young commutative algebraists, our interests in local cohomology were inspired by the influential work of Gennady Lyubeznik, and it is a pleasure to dedicate this paper to him.  The first author is partially  supported by NSA grant H98230-16-1-001. 

\end{acknowledgement}
\vspace{2mm}

\section{Preliminary results on the general  case}\label{prelim}

In this section we establish some general results on when $\lambda(\HH{i}{\fm}{R/I^n})<\infty$ for $n\gg 0$. We also give a connection between depth conditions on the Rees algebra and the growth rate of the length of local cohomology.  We first recall some familiar  notations.

\begin{Notation}\label{assum}
Let $(R,\fm)$ be a formally equidimensional local ring and a homomorphic image of a regular ring. Let $d=\dim(R)$. Let $I$ be an $R$-ideal, $\R:=\R(I)=R[It]$  the {\it Rees algebra} of $I$. 
Set $\ell:=\ell(I)=\dim \R/\fm\R$, the {\it analytic spread} of $I$. 
\end{Notation}

We now recall the following classical result.  

\begin{thm}[Grothendieck's Finiteness Theorem, {\cite[9.5.2]{BrSh}}]\label{GFT}
Let $R$ and $I$ be as in Notation  \ref{assum}. If  $M$  is a finitely generated $R$-module, then the least $i\in \NN$ for which $\HH{i}{I}{M}$ is not finitely generated is 
$$\min \{\depth M_\fp+\height(I+\fp)/\fp\mid \fp\in \Supp M\setminus V(I)\}.$$
\end{thm} 

In the next proposition we show that the if the $\limsup$ of $\{\frac{\lambda(\HH{i}{\fm}{I^n})}{n^d}\}_{n\gs 0}$ is positive then the ring $\HH{i}{\fm\R}{\R}$ cannot be Noetherian. This sequence is related to  $\{\frac{\lambda(\HH{i-1}{\fm}{R/I^n})}{n^d}\}_{n\gs 0}$. Indeed if $\depth R>i$, then from the long exact sequence of local cohomology we obtain $\HH{i-1}{\fm}{R/I^n})\cong \HH{i}{\fm}{I^n}$ for every $n$.


A Noetherian ring $S$ is said to satisfy {\it Serre's condition} $S_t$, for $t\in \NN$ if $\depth S_\fp\gs \min\{\height \fp, t\}$ for every $\fp\in \Spec (S)$.

\begin{lemma}\label{Noeth}
Let $R$ and $I$ be as in Notation  \ref{assum}. If $\R$ satisfies condition $S_{d-\ell}$, then the modules $\HH{i}{\fm\R}{\R}$ are Noetherian for $0\ls i\ls d-\ell$.
\end{lemma}
\begin{proof}
Let $\fp\in \Spec \R\setminus V(\fm\R)$ and let $h=\height \fp$. Since $\height \fm\R= d+1-\ell$, we have $\height (\fm\R+\fp)/\fp\gs \max\{ d+1-\ell-h, 1\}$. On the other hand, by assumption we have $\depth \R_{\fp}\gs\min\{h, d-\ell\}$. The result now follows by Theorem \ref{GFT}. 
\end{proof}


\begin{prop}\label{H}
Let $R$ and $I$ be as in Notation  \ref{assum}.
Fix $i\gs 1$ and consider the following statements:
\begin{enumerate}
\item $\limsup_{n\rightarrow \infty} \frac{\lambda(\HH{i}{\fm}{I^n})}{n^d} > 0$,
\item $\HH{i}{\fm\R}{\R}$ is not Noetherian,
\item $\ell \gs d-i+1 $, provided $\R$ satisfies $S_{i}$.
\end{enumerate}
Then $(1)\Rightarrow (2)\Rightarrow (3)$.

\noindent Moreover,  if $\HH{i}{\fm\R}{\R}$ is Noetherian, then the sequence $\{\lambda(\HH{i}{\fm}{I^n})\}_{n\gs 0}$ coincides with a polynomial of degree at most $\ell-1$.
\end{prop}
\begin{proof}
 Set $H:=\HH{i}{\fm\R}{\R}=\HH{i}{\fm}{\R}=\oplus_{n=0}^\infty \HH{i}{\fm}{I^n}$
 
 (1) $\Rightarrow$ (2):   If $H$ is Noetherian, then $H$ is a finitely generated $\R/(\fm\R)^N$-module for $N\gg 0$. Since $[\R/(\fm\R)^N]_0=R/\fm^N$ is Artinian, the Hilbert function of $H$ asymptotically coincides with a polynomial of degree $\dim(H)-1\ls \ell - 1< d$. Therefore, $\limsup_{n\rightarrow \infty} \frac{\lambda(\HH{i}{\fm}{I^n})}{n^d} =0$. 
 
  (2) $\Rightarrow$ (3): It follows from Lemma \ref{Noeth} that if $\ell \ls d-i$ then $H$ is Noetherian.
\end{proof}
  
\begin{remark}\label{reH}
In the case $i=1$, the three conditions in Proposition \ref{H} are equivalent. Indeed, in \cite[4.4]{UV}) the authors proved that over arbitrary local rings,  the $\limsup$ of $ \{\frac{\lambda(\HH{0}{\fm}{R/I^n})}{n^d}\}_{n\gs 0}$ is nonzero if and only if $\ell=d$. We do not know whether these conditions are equivalent for $i>1$. In Proposition \ref{bipH1pos} we give an example of a family of ideals for which $\R$ satisfies $S_2$, $\ell=d-1$, and $\limsup_{n\rightarrow \infty} \frac{\lambda(\HH{2}{\fm}{I^n})}{n^d} > 0$.
\end{remark}

\vspace{2mm}

A natural question to ask is whether $i>0$ and $\limsup_{n\rightarrow \infty} \frac{\lambda(\HH{i}{\fm}{R/I^n})}{n^d} = 0$ implies $\lambda(\HH{i}{\fm}{R/I^n})=0$ for every $n\gg 0$. In the next example  we show this is not the case. 

\begin{example}\label{nonMaxGrowth}
Let $I_1=(x^4, x^3y, xy^3, y^4, x^2y^2z ,w^2)$  and $I_2=(a,b)$ in the polynomial ring $R=k[x,y,z,w,a,b]$. If $I=I_1\cdot I_2$, then 
$$\lim_{n\rightarrow \infty} \frac{\lambda(\HH{1}{\fm}{R/I^n})}{n^2} \in \QQ_{>0},$$
and therefore $\limsup_{n\rightarrow \infty} \frac{\lambda(\HH{1}{\fm}{R/I^n})}{n^6} = 0$ while $\lambda(\HH{1}{\fm}{R/I^n})> 0$ for infinitely many $n$.
\end{example}
\begin{proof}
For every $n\gs 1$, consider the exact sequence 
$$0\rightarrow R/I^n=R/I_1^n\cap I_2^n\rightarrow R/I_1^n\oplus R/I_2^n\rightarrow R/I_1^n+I_2^n\rightarrow 0.$$
Clearly $\depth R/I_1^n> 1$  and $\depth R/I_2^n>1$, therefore  $\HH{1}{\fm}{R/I^n}=\HH{0}{\fm}{R/I_1^n+I_2^n}$ for every $n$. Let $R_1=k[x,y,z,w]$, $R_2=k[a,b]$, and $\fm_1$, $\fm_2$ their corresponding irrelevant maximal ideals. Since the support of the ideals $I_1$ and $I_2$ are disjoint, we have 
$$\lambda(\HH{1}{\fm}{R/I^n})=\lambda(\HH{0}{\fm_1}{R_1/I_1^n})\cdot \lambda(\HH{0}{\fm_2}{R_2/I_2^n})= \lambda(\HH{0}{\fm_1}{R_1/I_1^n})\cdot \frac{n(n+1)}{2}.$$
Hence, it suffices to show $\lambda(\HH{0}{\fm_1}{R_1/I_1^n})$ is a non-zero constant for $n\gg 0$.  
 
 Set $\G=\oplus_{i\gs 0} I_1^i/I_1^{i+1}$ and $H:=\HH{0}{\fm_1\G}{\G}=\oplus_{i\gs 0} \HH{0}{\fm_1}{I_1^i/I_1^{i+1}}$. Computations by  Macaulay2 \cite{GS} show that $\dim H=1$ and then $\lambda( \HH{0}{\fm_1}{I_1^i/I_1^{i+1}})$ agrees with a constant $C\neq 0$ for  $n\gg 0$. Now, the analytic spread of $I_1$ is $3$, then by \cite[5.4.6]{HS} and \cite[7.58]{Vas} we have $(I_1^n)^{sat}:= (I_1^n:\fm_1^\infty) \subseteq I_1^{n-4}$. Therefore, $\HH{0}{\fm_1}{R_1/I_1^{n}}= (I_1^n)^{sat}\cap I_1^{n-4}/I_1^{n}= \HH{0}{\fm_1}{I_1^{n-4}/I_1^{n}}$ and it follows that the sequence $\{\HH{0}{\fm_1}{R_1/I_1^{n}}\}_{n\gs 0}$ coincides with a polynomial for $n\gg 0$. Finally, from the following inequalities we observe this polynomial must be a non-zero constant
\begin{align*}
C=\lambda( \HH{0}{\fm_1}{I_1^{n-1}/I_1^{n}})\ls \lambda( \HH{0}{\fm_1}{R_1/I_1^{n}})=\lambda( \HH{0}{\fm_1}{I_1^{n-4}/I_1^{n}})\ls \oplus_{i=n-4}^{n-1}\lambda( \HH{0}{\fm_1}{I_1^{i}/I_1^{i+1}})=4C.
\end{align*}

\end{proof}

\subsection{Finite local cohomology}

\

One of the complications that arises if we want to study the asymptotic behavior of $\lambda(\HH{i}{\fm}{R/I^n})$ for $i>0$ is that, unlike the case $i=0$, the modules $\HH{i}{\fm}{R/I^n}$ need not be Noetherian and therefore their lengths may be infinite. In Proposition \ref{anotherGFT} we give sufficient conditions for the related lengths $\lambda(\HH{i}{\fm}{I^n})$ to be finite for $n\gg 0$.

\begin{prop}\label{anotherGFT}
Let $R$ and $I$ be as in Notation \ref{assum} and assume $R$ is  Cohen-Macaulay. Fix $t\in \NN$, then $\lambda(\HH{i}{\fm}{I^n})<\infty$ for every $i\ls t$ and $n\gg 0$ if and only if $$d-t\gs \height \fp-\lim_{n\rightarrow \infty}\depth_{R_\fp}(R/I^n)_\fp \text{ for every }\fp \neq \fm.$$
\end{prop}
\begin{proof}
By Theorem \ref{GFT} we have that for a any ideal $J$, the modules $\HH{i}{\fm}{J}$ are finitely generated for $i\ls t$ if and only if $$\min\{\depth_{R_\fp}J_{\fp}+\dim R/\fp\mid \fp\neq \fm\}\gs t+1.$$
By Brodmann's Theorem (see \cite[Theorem 2]{Br}) the sequence $\depth (R/I^n)_{\fp}$ is constant for $n\gg 0$, and  since $R$ is Cohen-Macaulay we have $\depth I^n_{\fp}=\depth (R/I^n)_{\fp}+1$ for $n\gg 0$. Therefore $\lambda(\HH{i}{\fm}{I^n})<\infty$  for $i\ls t$ and $n\gg 0$ if and only if $\depth_{R_\fp}(R/I^n)_\fp + 1 + d -\height \fp \gs t+1 \text{ for every }\fp \neq \fm$ and $n\gg 0$, which finishes the proof.
\end{proof}

\begin{cor}\label{corH2} With the assumptions of Proposition \ref{anotherGFT}, we have  
\begin{enumerate}
\item[(a)] $\HH{1}{\fm}{I^n}$ is always finitely generated for $n\gg 0$.
\item[(b)] $\HH{2}{\fm}{I^n}$ is finitely generated for $n\gg 0$ if and only if $\Ass (R/I^n)$ does not contain prime ideals of dimension one for $n\gg 0$.
\end{enumerate}
\end{cor}

\section{Monomial Ideals}\label{smono1}

For the rest of the paper we focus on the situation where $R=k[x_1,\ldots, x_d]$ and $\fm=(x_1,\ldots, x_d)$. In this section, we shall study in detail the case of $I$ being a monomial ideal.  This case is both important for later applications and possesses many interesting combinatorial aspects in its own right.

\subsection{Takayama's formula}

For every $F\subseteq [d]:= \{1,\ldots, d\}$, let $$\pi_F: R\longrightarrow R,$$ defined by $\pi_F(x_i)=1$ if $i\in F$, and $\pi_F(x_i)=x_i$ otherwise. For an $R$-ideal $I$, we write $I_F:=\pi_F(I)$. The following lemma presents some basic properties of the map $\pi_F$.

\begin{lemma}\label{interFs}
Let $I$ and $J$ be two monomial ideals. Then for every $F\subseteq [d]$ we have 
\begin{enumerate}
\item $(I\cap J)_F=I_F\cap J_F.$ 
\item Let $I= \cap_i Q_i$ be a primary decomposition of $I$, then $$I_F=\bigcap_{\supp(\sqrt{Q_i})\cap F=\emptyset}Q_i.$$
\end{enumerate}
\end{lemma}

 For a vector $\fa=(a_1,\ldots, a_d)\in \ZZ^d$, we write $\fx^\fa := x_1^{a_1}\cdots x_d^{a_d}$.   Let $G_\fa=\{i\mid a_i<0\}$. We also set $\fa^+=(a_1^+,\ldots, a_d^+)$, where $a_i^+=a_i$ if $i\not\in G_\fa $ and $a_i^+=0$ otherwise. 
 
 Let $I$ be a monomial ideal and  $\Delta_{\fa}(I)$ be the simplicial complex of all subsets $F$ of $[d]\setminus G_\fa$ such that $\fx^{\fa^+}\not\in I_{F\cup G_\fa}$.  It is a fact that $\Delta_{\fa}(I)$ is a subcomplex of $\Delta(I)$, the simplicial complex for which $\sqrt{I}$ is the Stanley-Reisner ideal (see \cite[1.3]{MT}). 
 
 \begin{thm}[{\cite[Theorem 1]{Tak}}]\label{Takayama}
 For every $\fa\in \ZZ^d$ and $i\gs 0$ we have
 $$\dim_k \HH{i}{\fm}{R/I}_{\fa} =  \dim_k  \tilde{H}_{i-|G_\fa |-1}(\Delta_{\fa}(I), k)$$
 \end{thm}
 
 It is important to remark that in \cite{Tak}, Theorem \ref{Takayama} is stated with extra conditions for $\fa$, however following its proof one can observe that if any of those conditions is violated then $\dim_k \tilde{H}_{i-|G_\fa |-1}(\Delta_{\fa}(I), k)=0$. For more information see \cite{Tak} and \cite{MT}.
 
 \begin{remark}\label{otherInt}
 Let $R_{\fa+}=k[x_i\mid i\not\in G_\fa]$, $\fm_{\fa+}$ be the irrelevant ideal of $R_{\fa+}$, and $I_{\fa+}=I_{G_\fa}\cap R_{\fa+}$. Then it follows directly from Theorem \ref{Takayama} that 
 $$\dim_k \HH{i}{\fm}{R/I}_{\fa} =  \dim_k \HH{i-|G_\fa|}{\fm_{\fa+}}{R_{\fa+}/I_{\fa+}}_{\fa^+},$$
 where we are identifying ${\fa^+}$ with the corresponding vector in $\NN^{d-|G_\fa|}$.
 \end{remark}
 


We say that  $(I_n)_{n\gs 0}$ is a {\it graded family} of ideals if $I_0=R$ and for every $n,m\gs 0$ we have $I_nI_m\subseteq I_{m+n}$. We say that the graded family $(I_n)_{n\gs 0}$  is {\it Noetherian} if $\R(t)=\oplus_{n\gs 0} I_n t^n$ is a finitely generated algebra. 

\begin{example}
Let $I$ be a monomial ideal, the following are examples of Noetherian graded families of monomial ideals.
\begin{enumerate}
\item The regular powers $(I^n)_{n\gs 0}$.
\item The  saturations $(\tilde{I^n})_{n\gs 0}$, where $\tilde{I^n}=I^n:\fm^{\infty}$.
\item The symbolic powers $(I^{(n)})_{n\gs 0}$.
\item More generally, for any monomial ideal $J$, the ideals $(I^n:J^{\infty})_{n\gs 0}$ (see \cite[3.2]{HHT}).
\item The integral closures $(\overline{I^n})_{n\gs 0}$ (see \cite[7.58]{Vas}).
\end{enumerate}
\end{example}

 Let $n$ be a variable. A {\it quasi-polynomial} on $n$ is a function $f:\NN\rightarrow \QQ$ for which there exist $\pi \in \NN$ and polynomials $p_i(n)\in \QQ[n]$ for $i=0,\ldots, \pi-1$ such that $f(n)=p_i(n)$ whenever $n\equiv i\, (\text{mod } \pi)$.

\begin{prop}{\label{quasi-pol_GenM}}
Let $(I_{1,n})_{n\gs 0},\ldots, (I_{r,n})_{n\gs 0}, (J_{1,n})_{n\gs 0},\ldots,(J_{s,n})_{n\gs 0}$ be Noetherian graded families of monomial ideals. For every $n\gs 1$ consider the set
  $$S_n = \{\fa\in\NN^d \mid \fx^\fa \in I_{i,n} \text{ for every }1\ls i\ls r, \text{ and }\fx^\fa \not\in J_{j,n} \text{ for every } 1\ls j\ls s\}.$$
Let $f(n)=|S_n|$ for every $n\gs 1$ and assume $f(n)<\infty$ for $n\gg 0$, then $f(n)$ agrees with a quasi-polynomial  of degree at most $d$ for $n\gg0$.
\end{prop}
For the proof of this proposition we need to introduce some notation.

A {\it Presburger formula} $F$ is a first order formula that can be written using quantifiers $\exists$,  $ \forall$,  boolean operations $\wedge$, $\vee$, $\neg$, and linear inequalities  with variables evaluated over the natural numbers. The free variables of $F$ is the set of variables $\fa$ not affected by quantifiers.  Let $\fa$ be a set of variables and $F(\fa, n)$ a  Presburger formula having $\fa$ and $n$ as free variables. Set 
$$f(n) = \#\{\fa\in \NN^d \mid F(\fa, n)\},$$
this function is called the {\it Presburger counting function} for $F(\fa, n)$.

A {\it piecewise quasi-polynomial} on $n$ is a function $f:\NN\rightarrow \QQ$ such that there exists a finite  partition $\{\Gamma_i\cap \NN\} _i$ of $\NN$ with $\Gamma_i$ intervals and such that there exist quasi-polynomials $f_i$ with $f(n) = f_i(n)$ whenever $n\in \Gamma_i\cap\NN$, for every $i$.

We say that a sequence $\{s_n\}_{n\gs 0}$ has a {\it rational generating function} if $\sum_{n\in \NN} s_nx^n$ can be written as $\frac{q(x)}{(1-x^{b_1})(1-x^{b_2})\cdots (1-x^{b_r})}$ for some $q(x)\in \QQ[x]$ and positive integers $b_1,b_2,\ldots, b_r$.

\begin{thm}[
{\cite[1.10, 1.12]{Woods}}]\label{kevin}
Given a function $f:\NN\rightarrow \QQ$, the following are equivalent:
\begin{enumerate}
\item $f$ is a Presburger counting function.
\item $f$ is a piecewise quasi-polynomial whose range is in $\NN$.
\item The generating function $\sum_{n\in \NN}f(n)x^n$ is rational.
\end{enumerate}
\end{thm}

\vspace{2mm}

\begin{proof}[Proof of Proposition \ref{quasi-pol_GenM}]

Let $N$ be such that $f(n)<\infty$ for every $n\gs N$  and such that all the algebras $\oplus_{n\gs 0} I_{i,n}t^n$ for $i=1,\ldots, r$ and  $\oplus_{n\gs 0} J_{i,n}t^n$ for $i=1,\ldots, s$ are generated up to degree $N$.

Fix $\fa = (a_1,\ldots, a_d)$ . Let $\fx^{\fu_{i,p,1}}, \ldots, \fx^{\fu_{i,p, m_{i,p}}}$ be the minimal monomial generators of $I_{i,p}$ for $i=1,\ldots, r$ and $1\ls p\ls N$, and consider the following Presburger formula 
$$Y_i(\fa, n) = (\exists t_{p,q}, 1\ls p\ls N,\, 1\ls q \ls m_{i,p}) \Big( \sum_{p,q} pt_{p,q}= n \wedge \fa \succeq \sum_{p,q} t_{p,q} \fu_{i, p,q} \Big),$$
where $\succeq$ is indicates component-wise inequality.  Similarly, let 
$\fx^{\fv_{i,p,1}}, \ldots, \fx^{\fv_{i,p, n_{i,p}}}$  be the minimal monomial generators of $J_{i,p}$ for $i=1,\ldots, s$, and $1\ls p\ls N$, and consider
$$Z_i(\fa, n) = (\forall t_{p,q}, 1\ls p\ls N,\, 1\ls q \ls n_{i,p} )\Big( \sum_{p,q} pt_{p,q} < n \vee \neg(\fa \succeq \sum_{p,q} t_{p,q} \fv_{i, p,q} )\Big).$$

It follows that for every $i$ and $n$,  $\fx^{\fa}\in I_{i,n} $ if and only if $Y_i(\fa, n)$ is satisfied, and  $\fx^{\fa}\not\in J_{i,n} $ if and only if $Z_i(\fa, n)$ is satisfied. 
Hence, $$g(n):=f(n+N) = \#\left\{\fa\in\NN^d\mid \Big(\bigwedge_{i=1}^{r} Y_i(\fa,n+N)\Big) \wedge \Big(\bigwedge_{i=1}^{s} Z_i(\fa,n+N)\Big) \right\}$$
is a Presburger counting function, and by Theorem \ref{kevin} a piecewise quasi-polynomial. This implies that $g$, and hence $f$, agree with a quasi-polynomial for $n\gg0$, as desired.

Finally, an analysis of the proof  of Theorem \ref{kevin} shows that this  quasi-polynomial has degree at most $d$. 
\end{proof}

\begin{lemma}\label{bt}
Let $(I_{n})_{n\gs 0}$ be a Noetherian graded family of monomial ideals and fix $\Delta'$ a simplicial complex with support $[d]$. Consider the function 
$$f_{\Delta'}(n) = \#\{\fa\in \NN^d\mid \Delta_{\fa}(I_n)=\Delta'\},$$
and assume $f_{\Delta'}(n)<\infty$ for $n\gg0$. Then $f_{\Delta'}(n)$ agrees with a quasi-polynomial  of degree at most $d$ for $n\gg0$.
\end{lemma}
\begin{proof}
It is clear that $((I_{n})_F)_{n\gs 0}$  is a Noetherian graded family for every $F\subseteq [d]$, therefore the result follows from the description of $\Delta_\fa(I_n)$ and  Proposition \ref{quasi-pol_GenM}.
\end{proof}

We are now ready to present the main theorem of this section.

\begin{thm}\label{mainMono}
Let $(I_{n})_{n\gs 0}$ be a Noetherian graded family of monomial ideals.  Assume $\lambda(\HH{i}{\fm}{R/I_n})<\infty$ for $n\gg 0$. Then the sequence $\{\lambda(\HH{i}{\fm}{R/I_n})\}_{n\gs 0}$ agrees with a quasi-polynomial of degree at most $d$ for $n\gg 0$. Moreover, this sequence has a rational generating function. 
\end{thm}

\begin{proof}
Consider $n\gg 0$ and fix $\fa$ such that $\HH{i}{\fm}{R/I_n}_{\fa}\neq 0$. Since $\lambda(\HH{i}{\fm}{R/I_n})<\infty$ we have $\fa\in \NN^d$ (see \cite[Proposition 1]{Tak}). We remark that since $I_1^n\subseteq I_n$, then $\sqrt{I_1}\subseteq \sqrt{I_n}$. Therefore, $\Delta(I_n)$ is a subcomplex of $\Delta(I_1)$ and then so is $\Delta_{\fa}(I_n)$.

Now,
\begin{align*}
\lambda(\HH{i}{\fm}{R/I_n}) &= \sum_{\fa\in \NN^d} \dim_k \HH{i}{\fm}{R/I_n}_{\fa}\\
&= \sum_{\fa\in \NN^d} \dim_k  \tilde{H}_{i-1}(\Delta_{\fa}(I_n), k),\,\, \text{ by Theorem \ref{Takayama}},\\
&= \sum_{\Delta'\subseteq \Delta(I_1)} \sum_{\{ \fa \mid \Delta_{\fa}(I_n) = \Delta' \}} \dim_k  \tilde{H}_{i-1}(\Delta', k)\\
&= \sum_{\Delta'\subseteq \Delta(I_1)} \dim_k  \tilde{H}_{i-1}(\Delta', k)f_{\Delta'}(n),
\end{align*}
where $f_{\Delta'}(n)$ is defined as in Lemma \ref{bt}. The conclusion now follows from Lemma \ref{bt} as a linear combination of quasi-polynomials is a quasi-polynomial. The last statement follows from Theorem \ref{kevin}.
\end{proof}

\begin{cor}\label{corofMainMon}
Under the assumptions of Theorem \ref{mainMono}, we have
$$\limsup_{n\rightarrow \infty}\frac{\lambda(\HH{i}{\fm}{R/I_n})}{n^d}<\infty.$$
\end{cor}

\begin{example}\label{P5example}
Let $I=(x_1x_2,x_2x_3,x_3x_4,x_4x_5)$. Let $\Delta_1$ and $\Delta_2$ be the simplexes in the sets $S_1=\{1,3,5\}$ and $S_2=\{2,4\}$ respectively. It follows from Theorem \ref{Takayama} that for every $n\gs 1$, $\HH{1}{\fm}{R/I^n}_{\fa}\neq 0$ if and only if  $\Delta_\fa(I)=\Delta_1 \sqcup \Delta_2$ if and only
\begin{align}\label{ineqs}
\nonumber a_1 + a_3 + a_4&\gs n\\
a_2 + a_3 + a_5&\gs n\\
\nonumber a_1 + a_3 + a_5 &< n\\
\nonumber a_2 + a_4 &< n.
\end{align}

Set $\delta_1(\fx)=x_1 + x_3 + x_4$ , $\delta_2(\fx)=x_2 + x_3 + x_5$, $\delta_3(\fx)=x_1 + x_3 + x_5 $, and $\delta_4(\fx)=x_2 + x_4$. Let us abbreviate the set of inequalities $\{x_1\gs 0,\,x_2\gs 0,\,x_3\gs 0,\, x_4\gs 0,\, x_5\gs 0 \}$ by $\fx \succeq \f0$ and define the following polytopes in $\RR^5$,
$$\P_1=\{\fx\succeq \f0,\,\delta_1\gs 1,\, \delta_2\gs 1, \delta_3\ls 1, \delta_4\ls 1 \},\quad 
\P_2=\{\fx\succeq \f0,\,\delta_1\gs 1,\, \delta_2\gs 1, \delta_3=1, \delta_4\ls 1 \},$$
$$\P_3=\{\fx\succeq \f0,\,\delta_1\gs 1,\, \delta_2\gs 1, \delta_3\ls1, \delta_4= 1 \},\quad 
\P_4=\{\fx\succeq \f0,\,\delta_1\gs 1,\, \delta_2\gs 1, \delta_3=1, \delta_4=1 \}.$$
By the inclusion-exclusion principle we can count the number of integer points $\fa$ satisfying the inequalities \eqref{ineqs} as 
\begin{equation}\label{incexcl}
E(n)=\#(n\P_1\cap \NN^d)-\#(n\P_2\cap \NN^d)-\#(n\P_3\cap \NN^d)+\#(n\P_4\cap \NN^d),
\end{equation}
that is, a combination of the  Ehrhart quasi-polynomials of the rational polytopes $\P_1,\, \P_2,\, \P_3,$ and $ \P_4,$ (see \cite{Mc} and \cite{St}).  Using the function {\tt HilbertSeries} of {\it Normaliz} (see \cite{Normaliz}) we obtain 
$$
 \lambda(\HH{1}{\fm}{R/I^n})=E(n)=
\begin{cases}
\frac{n^5}{240}+\frac{n^4}{96}-\frac{n^3}{48}-\frac{n^2}{24}+\frac{n}{60},\,\,\, \text{if }n\equiv 0 \Mod{2}\\
\frac{n^5}{240}+\frac{n^4}{96}-\frac{n^3}{48}-\frac{n^2}{24}+\frac{n}{60} + \frac{1}{32},\,\,\, \text{if }n\equiv 1 \Mod{2};
\end{cases}
$$
and the generating function
$$\sum_{n=0}^\infty \lambda(\HH{1}{\fm}{R/I^n})x^n = \frac{x^3}{(1-x)^5(1-x^2)}.$$
\end{example}

\vspace{2mm}

\section{Existence of  limit for classes of monomial ideals}\label{smono2}

Let $R=k[x_1,\ldots, x_d]$ and $\fm=(x_1,\ldots, x_d)$. For a given $R$-ideal $I$, we denote by $\ol{I}$ the integral closure of $I$ and by $\conv(I)$ the {\it Newton polyhedron} of $I$, i.e., the convex hull of the set $\{\fa\in \NN^d\mid \fx^\fa\in I\}$. Notice that $\NN^d\cap \conv(I)=\{\fa\in \NN^d\mid \fx^\fa\in \ol{I}\}$


We start with the following technical lemma that will allow us to prove the existence of the limit in the main theorem.

\begin{lemma}\label{coconvex}
Let $\Gamma, \Gamma_1,\ldots, \Gamma_s$ be convex polyhedra in $\RR^d$ such that $\Gamma_i\subseteq \Gamma$ for $1\ls i\ls s$. Set $n\C:=n\Gamma\setminus \cup_{i=1}^s n\Gamma_i$ for every $n\gs 1$ and assume  $\C$ $($equivalently every $n\C )$ is bounded, then
$$\vol(\C)=\lim_{n\rightarrow \infty} \frac{\#(\ZZ^d\cap n\C)}{n^d}.$$
\end{lemma}
\begin{proof}
Set $c:=\dim \C$ and let $\P$ be a $c$-dimensional polytope that contains $\C$, and $E_\P(n)$ its Ehrhart polynomial. From $\#(\ZZ^d\cap n\C)\ls E_\P(n)$ it follows that if $c<d$, then $\lim_{n\rightarrow \infty} \frac{\#(\ZZ^d\cap n\C)}{n^d}=0=\vol(\C)$, hence we may assume $c=d$. Now, $$\frac{\#(\ZZ^d\cap n\C)}{n^d}=\#\Big(\frac{1}{n}\ZZ^d\cap \C\Big) \Big(\frac{1}{n}\Big)^d$$ 
and the latter is a Riemann sum of the characteristic function $\chi_\C(\fx)$ using the partition of $\RR^d$ consisting of the boxes with vertices on $\frac{1}{n}\ZZ^d$. Therefore, 
$$\lim_{n\rightarrow \infty} \frac{\#(\ZZ^d\cap n\C)}{n^d} = \int_{\RR^d} \chi_\C(\fx) \, d\fx = \vol(\C).$$
\end{proof}

 \vspace{1mm}
 
\begin{lemma}\label{limIdeanAvoid}
Let $R=k[x_1,\ldots, x_d]$ and $I_1,\ldots, I_r, J_1,\ldots, J_s$ monomial ideals. For every $n\gs 1$ consider the set
  $$S_n = \{\fa\in\NN^d \mid \fx^\fa \in \ol{I_i^n} \text{ for every }1\ls i\ls r, \text{ and }\fx^\fa \not\in \ol{J_i^n} \text{ for every } 1\ls i\ls s\}.$$
Assume $|S_n|<\infty$ for $n\gg 0$, then the limit $\lim_{n\rightarrow \infty}\frac{|S_n|}{n^d}$ exists and is equal to a rational number.
\end{lemma}
\begin{proof}
Let $\Gamma:=\cap_{i=1}^r \conv(I_i)$ and $\Gamma_i:=\Gamma \cap \conv(J_i)$ for every $1\ls i\ls s$ (see Example \ref{exofPol}).  Notice that $\fa\in S_n$ if and only if $\fa\in n\C:=n\Gamma\setminus \cup_{i=1}^s n\Gamma_i$ and by assumption $\C$ is bounded. The result follows by Lemma \ref{coconvex}.
\end{proof} 

\vspace{2mm}
\begin{example}\label{exofPol}
Let $R=k[x,y]$, $I_1=(xy^5, x^4y^4, x^5y)$, $J_1=(xy^8, x^3y^6)$, $J_2=(x^6y^4, x^7y^3)$, and $J_3=(x^9y, x^{10})$. Following the notation in Lemma \ref{limIdeanAvoid}, in Figure \ref{Fig} we can observe that $\cup_{i=1}^3 n\Gamma_i$ covers the yellow unbounded region, $\C$ is the green bounded region, and
$\Gamma$ is the union of these two. We also note that the point $(4,4)$ is in the interior of $\conv(I_1)$ though it corresponds to a minimal generator of $I_1$. The dotted segment represents the bounded facet of $\conv(J_3)$.
\end{example}

\begin{figure}
\centering
\includegraphics[scale=.3]{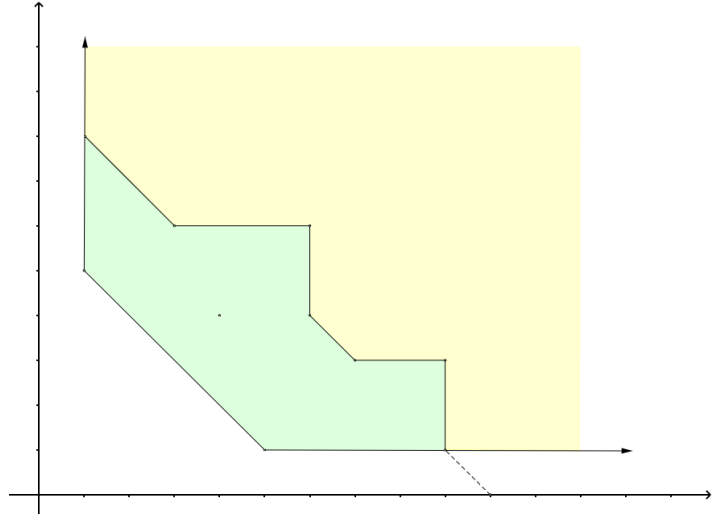}
\caption{Regions corresponding to the ideals in Example \ref{exofPol}.} 
 \label{Fig}
\end{figure}

We are now ready to present the main theorem of this section.

\begin{thm}\label{mainNormal}
 Let $I$ be a monomial ideal.  Assume $\lambda(\HH{i}{\fm}{R/\ol{I^n}})<\infty$ for $n\gg 0$, then the limit
$$\lim_{n\rightarrow \infty}\frac{\lambda(\HH{i}{\fm}{R/\ol{I^n}})}{n^d}$$
exists and is a rational number.
\end{thm}
\begin{proof}
 Proceeding as in Theorem \ref{mainMono} we obtain 
$$\lambda(\HH{i}{\fm}{R/\ol{I^n}}) = \sum_{\Delta'\subseteq \Delta(I)} \dim_k  \tilde{H}_{i-1}(\Delta', k)f_{\Delta'}(n).$$
where $f_{\Delta'}(n) = \#\{\fa\in \NN^d\mid \Delta_{\fa}(\ol{I^n})=\Delta'\}.$
We now show that for each $\Delta'$ such that $\tilde{H}_{i-1}(\Delta', k)\neq 0$ the limit $\lim_{n\rightarrow \infty}\frac{f_{\Delta'}(n)}{n^d}$ exists and is rational, which will finish the proof. Indeed, by assumption for any such $\Delta'$ we must have $f_{\Delta'}(n)<\infty$ for $n\gg 0$. Moreover, we have that $\Delta_{\fa}(\ol{I^n}) = \Delta'$ if and only if $\fx^\fa\not \in \ol{I^n}_F$, for every facet $F$ of $\Delta'$, and $\fx^\fa\in \ol{I^n}_G$, for every minimal non-face $G$ of $\Delta'$. The theorem then follows by Lemma \ref{limIdeanAvoid} since $\ol{I^n}_F=\ol{I^n_F}$ and $\ol{I^n}_G=\ol{I^n_G}$ for every $F$ and $G$ (see \cite[1.1.4(2)]{HS}).

\end{proof}

A monomial  ideal $I$ is {\it locally normal} if $I_{\{i\}}$ is normal for every $i$, i.e., $ I_{\{i\}}^n=\ol{I_{\{i\}}^n}$ for every $i$ and $n$.

\begin{cor}\label{corLocNorm}
 Let $I$ be a locally normal monomial ideal. Assume $\lambda(\HH{i}{\fm}{R/I^n})<\infty$ for $n\gg 0$, then the limit
$$\lim_{n\rightarrow \infty}\frac{\lambda(\HH{i}{\fm}{R/I^n})}{n^d}$$
exists and is a rational number. 
\end{cor}
\begin{proof}
The limit exists for $i=0$ by \cite[2.6]{HPV}, then we can assume $i\gs 1$. By assumption we have $\ll(\ol{I^n}/I^n)<\infty$ for every $n\gs 1$, therefore $\HH{i}{\fm}{R/I^n}\cong \HH{i}{\fm}{R/\ol{I^n}}$. The conclusion now follows by Theorem \ref{mainNormal}.
\end{proof}

\begin{example}
From Example \ref{P5example} we observe that if $I=(x_1x_2,x_2x_3,x_3x_4,x_4x_5)$ then $I$ is normal by \cite{SVV} and $\lim_{n\rightarrow \infty}\frac{\lambda(\HH{1}{\fm}{R/I^n})}{n^5}=\frac{1}{240}.$
\end{example}

\vspace{3mm}

\subsection{Limit for edge ideals}
Let $\G$ be a graph. We denote by $V(\G)$ the vertex set of $\G$. For every subset $S$ of $V(\G)$ we denote by $\adj(S)$ the set of vertices not in $S$ that are adjacent to a vertex of $S$. We denote by $\st(S):=S\cup \adj(S)$, the {\it star} of $S$, and by $\G\setminus S$ the induced subgraph of $\G$ whose vertex set is $V(\G)\setminus S$.

Assume $d\gs 3$ and let $\G$ be a graph such that $V(\G)=[d]$. We denote by $I(\G)$ the {\it edge ideal} of $\G$, i.e., the monomial ideal generated by $\{x_ix_j\mid ij \text{ is an edge of }\G.\}$ 

For every $x$, let $\G_x$ be the graph obtained from $\G\setminus \st(x)$ by removing all of its isolated vertices. Notice that $I_{\{i\}}=I(\G_{x_i})+\sum_{x_j\in \adj(x_i)}(x_j)$ for every $1\ls i\ls d$. 

\begin{prop}\label{finiteLength}
Let $\G$ be a graph with no isolated vertices  and $I:=I(\G)$ its edge ideal, then $\lambda(\HH{1}{\fm}{R/I^n})<\infty$ for $n\gg 0$ if and only if for every $i\in [d]$ at least one of the following two conditions holds
\begin{enumerate}
\item[$(i)$]$\G\setminus \st(x_i)$ has  isolated vertices, or
\item[$(ii)$]  the graph $\G_{x_i}$ has at least one bipartite connected component.
\end{enumerate} 
\end{prop}

\begin{proof}
Since $d\gs 3$, from the exact sequence $0\rightarrow I^n \rightarrow R \rightarrow R/I^n \rightarrow0$ we obtain $\HH{1}{\fm}{R/I^n} \cong  \HH{2}{\fm}{I^n}$ for every $n\gs 1$. Hence, by Corollary \ref{corH2}, we have that $\lambda(\HH{1}{\fm}{R/I^n})<\infty$ for $n\gg 0$ if and only if the monomial prime ideal generated by $V_i:=V(\G)\setminus \{i\}$ is not an associated prime of $R/I^n$ for $n\gg 0$ and every $i$. By Lemma \ref{interFs} the latter holds if and only if the maximal ideal of $R_{\{i\}}:=k[\{x_j\mid j\neq i\}]$ is not an associated prime of $I(\G_{x_i})+\sum_{x_j\in \adj(x_i)}(x_j)$ for every $i$. By \cite[3.4]{LT} this is equivalent to $\G_{x_i}$ having bipartite components or $V(\G_{x_i})\cup \{j\mid x_j\in \adj(x_i)\}\neq [d]\setminus \{i\}$, i.e., $\G\setminus \st(x_i)$ having isolated vertices.
\end{proof}

\begin{cor}\label{locallyBipartite}
Let $\G$ be a graph with no isolated vertices  and $I:=I(\G)$ its edge ideal. Assume $\G_{x_i}$ is bipartite for every $i\in [d]$,  then the limit
$$\lim_{n\rightarrow \infty}\frac{\lambda(\HH{1}{\fm}{R/I^n})}{n^d}$$
exists and is a rational number.
\end{cor}
\begin{proof}
From Proposition \ref{finiteLength} it follows that $\lambda(\HH{1}{\fm}{R/I^n})<\infty$ for $n\gg 0$. By \cite{SVV} and the assumption the ideal $I$ is locally normal. Hence, the conclusion follows by Corollary \ref{corLocNorm}.
\end{proof}

\begin{prop}\label{bipH1pos}
Let $\G$ be a bipartite connected graph and $I:=I(\G)$ its edge ideal. Assume $\height I>1$, then 
$$\lim_{n\rightarrow \infty}\frac{\lambda(\HH{1}{\fm}{R/I^n})}{n^d}>0$$
exists and is a positive rational number.
\end{prop}
\begin{proof}
The condition on the height of $I$ guarantees $\G_x$ is not empty for any $x$, then by Corollary \ref{locallyBipartite} we know the limit exists and is a rational number. In order to show the limit is positive, it is enough to show there exists $\Delta'\subseteq \Delta(I)$ such that $\tilde{H}_{0}(\Delta', k)\neq 0$ and $f_{\Delta'}$ has degree $d$ (cf. proof of Theorem \ref{mainMono}). Let us split the elements in $[d]$ between the sets $S_1=\{1,\ldots, b\}$ and $S_2=\{b+1,\ldots, d\}$ such that every edge of $\G$ contains a vertex in $S_1$ and another one in $S_2$. Let $\Delta'$ be the simplicial complex whose facets are the simplexes $F_1$ and $F_2$ with vertex set $S_1$ and $S_2$ respectively. Since $\Delta'$ is disconnected we have $\tilde{H}_{0}(\Delta', k)\neq 0$.

The minimal non-faces of $\Delta'$ are $\{ij\}$ with $i\in S_1$ and $j\in S_2$, therefore $$f_{\Delta'}(n)=\#\{\fa\in \NN^d\mid \fx^\fa\in \bigcap_{i\in S_1, j\in S_2} I^n_{ij}, \,\fx^\fa\not\in I^n_{S_1},\, \fx^\fa\not\in I^n_{S_2} \}.$$
Let $P_1$ and $P_2$ be the monomial prime ideals with support $S_1$ and $S_2$ respectively. 
Since $I$ is {\it normally torsion free}, i.e., $\Ass(R/I^n)=\Ass(R/I)$ for every $n\gs 1$ (cf. \cite{SVV}), 
by Lemma \ref{interFs}(2) we have $\bigcap_{i\in S_1, j\in S_2} I^n_{ij}=\bigcap_{ P\in\Ass(R/I)\setminus \{P_1,P_2\}}  P^n=I^n:(P_1P_2)^\infty$. Hence,
 $$f_{\Delta'}(n)=\#\{\fa\in \NN^d\mid \fx^\fa\in I^n:(P_1P_2)^\infty, \, \fx^\fa\not\in P_1^n,\,\fx^\fa\not\in P_2^n \}.$$

For a monomial prime ideal $P$, let  $H_P$ be the hyperplane of $\RR^d$ defined by the equation $\sum_{i\in \supp(P)} x_i= 1.$ Let $\P$ be the polytope
$$\P :=H_{P_1}^- \cap H_{P_2}^- \cap  \bigcap_{ P\in\Ass(R/I)\setminus \{P_1,P_2\} } H_P^+, $$
hence $ f_{\Delta'}$ is the Ehrhart polynomial of $\P$ minus the ones of the faces $\P\cap H_{P_1}$ and $\P\cap H_{P_2}$. Therefore, $\deg f_{\Delta'}=d$ if and only if $\dim \P =d$. In order to prove $\dim \P=\dim d\P=d$ it suffices to show that there exists $\fu\in d\P$ such that $\fu+\fe_i\in d\P$ for every $1\ls i\ls d$. 
Let $\T$ be a spanning tree of $\G$ and set $\fx^\fu:=\prod_{ij\in  \T}x_ix_j$. Since $\T$ has $d-1$ edges, it follows that $\fu$ and each $\fu+\fe_i$ belong to  $dH_{P_1}^-$ and  $dH_{P_2}^-$. It remains to show $\fu\in \bigcap_{ P\in\Ass(R/I)\setminus \{P_1,P_2\} } dH_P^+$ or equivalently  $\fx^\fu\in  I^d:(P_1P_2)^\infty$. We show that in fact $P_1P_2\fx^\fu\subseteq I^d$.

Let $x_i\in P_1$ and $x_j\in P_2$ and $i=a_0, a_1,\ldots, a_{2k}, a_{2k+1}=j$ a path from $i$ to $j$ in $\T$. Notice that $\fx^{\fu'}:=(x_{a_0}x_{a_1})(x_{a_1}x_{a_2})\cdots (x_{a_{2k}}x_{a_{2k+1}})\in I^{2k+1}$ while 
$$(x_ix_j)\fx^{\fu'}=(x_{a_0}x_{a_1})^2(x_{a_2}x_{a_3})^2\cdots (x_{a_{2k}}x_{a_{2k+1}})^2\in I^{2k+2}.$$
Hence, $x_ix_j\fx^\fu\in I^d$ finishing the proof. 

\end{proof}

\vspace{3mm}

\section{Finiteness of lim sup}\label{lsup}


\vspace{3mm}

Let $R=k[x_1,\ldots, x_d]$ and $\fm=(x_1,\ldots, x_d)$. For a graded $R$-module $M=\oplus_{i\in \ZZ}M_i$ and $n\in \ZZ$, we use the notation $M_{\gs n}=\oplus_{i\gs n}M_i$ and $M_{\ls n}=\oplus_{i\ls n}M_i$. For $M\neq 0$, we also use $$\indeg{M}=\min\{i\mid M_i\neq0\},\qquad \topdeg{M}=\max\{i\mid M_i\neq 0\}.$$
 If $M=\oplus_{\fa\in \ZZ^d}M_\fa$ is $\ZZ^d$-grade and $\fb\in \ZZ^d$, then we set $M_{\succeq \fb}=\oplus_{\fa\succeq \fb}M_\fa$, where $\succeq$ denotes component-wise inequality. 

We recall that the {\it Castelnuovo-Mumford regularity} of a graded module $M$ is defined as
$$\reg{M}= \max\{\topdeg \HH{i}{\fm}{M} + i \mid 0\ls i\ls \dim M\}$$
which is a finite number.  

In the next lemma we denote by $\f0$ the vector in $\ZZ^d$ whose components are all zero. 

\begin{lemma}\label{geq0}
Let $\{I_n\}_{n\gs 1}$ be a family of monomial ideals. Assume there exists $\beta \in \NN$ such that $\reg(I_n)\ls\beta n$ for every $n$, then for every $i \gs 0$ we have
$$\limsup_{n\rightarrow \infty}\frac{\lambda(\HH{i}{\fm}{R/I_n}_{\succeq \f0}) }{n^d}<\infty.$$ 

\end{lemma}
\begin{proof}

By Theorem \ref{Takayama} we obtain
\begin{align*}
\lambda(\HH{i}{\fm}{R/I_n}_{\succeq \bf{0}}) &= \sum_{\fa\in \NN^d,\,|\fa|\ls\beta n} \dim_k \HH{i}{\fm}{R/I_n}_{\fa}\\
&=\sum_{\fa\in \NN^d,\,|\fa|\ls\beta n} \dim_k  \tilde{H}_{i-1}(\Delta_\fa, k)\\
&< \sum_{\Delta'} \dim_k  \tilde{H}_{i-1}(\Delta', k)(\beta n)^d,
\end{align*}
where the last sum runs over all the simplicial complexes on $[d]$. The result follows.
\end{proof}

\begin{lemma}\label{regGD}
Let $I$ be a monomial ideal, then for every $F\subseteq [d]$ we have $\reg I_F \ls \reg I$.
\end{lemma}
\begin{proof}
By induction it suffices to prove the statement when $|F|=1$. Assume $F=\{1\}$ and let $R_1=k[x_2,\ldots, x_d]$, $I_1=I_{\{1\}}$. Let $\FF$ be a minimal resolution of $R/I$. Since $x_1-1$ is a nonzero divisor in $R/I$ it follows that $\FF/(x_1-1)\FF$ is a resolution of $R_1/I_1\cap R_1 $, but possibly non-minimal.  Hence,  $\reg R/I\gs \reg R_1/I_1\cap R_1$ and  the statement follows because $\reg  I_1\cap R_1 = \reg I_1 $.
\end{proof}

The following is the main theorem of this subsection. 

\begin{thm}\label{gradedFiniteLimpup}
Let $R=k[x_1,\ldots,x_d]$ and $I$ be a homogeneous $R$-ideal,  then for every  $\alpha \in \ZZ$, we have $$\limsup_{n\rightarrow \infty}\frac{\lambda(\HH{i}{\fm}{R/I^n}_{\gs -\alpha n}) }{n^d}<\infty.$$
\end{thm}
\begin{proof}
By possible extending the field $k$, we can assume that $|k|=\infty$. Let $\varphi: R\rightarrow R$ an automorphism of $R$ given by a generic change of variables. Notice $\varphi$ induces graded isomorphisms $R/I^n\cong R/\varphi(I^n)$ and $\HH{i}{\fm}{R/I^n}\cong \HH{i}{\fm}{R/\varphi(I)^n}$. Consider $J_n:=\gin(I^n)=\ini(\varphi(I)^n)$, the {\it generic initial ideal} of $I^n$ with respect to a revlex order (cf. \cite[Chapter 4]{HH}). By \cite[(2.4)]{BS} we have $\reg{J_n}=\reg{I^n}$ and by \cite[1.1]{CHT} (see also \cite[Theorem 5]{Kod}), there exists $\beta$ such that $\reg{J_n}\ls \beta n$ for every $n$.  Now, 
\begin{align*}
\lambda(\HH{i}{\fm}{R/I^n}_{\gs -\alpha n}) &= \lambda(\HH{i}{\fm}{R/\varphi(I)^n}_{\gs -\alpha n})\\
&\ls \lambda(\HH{i}{\fm}{R/J_n}_{\gs -\alpha n}),\,\,\, \text{by \cite[2.4]{Sba},}\\
&\ls \sum_{-\alpha n \ls |\fa| \ls\beta n} \dim_k \HH{i}{\fm}{R/J_n}_{\fa}.
\end{align*}
Recall that a sequence of real numbers $a_n$ is $O(n^t)$ if $\limsup_{n\rightarrow \infty}\frac{|a_n|}{n^t}<\infty$. Fix $S\subseteq [d]$, we will prove that 
$$\sum_{-\alpha n \ls |\fa| \ls\beta n, G_\fa=S} \dim_k \HH{i}{\fm}{R/J_n}_{\fa}=O(n^d)$$ 
which will finish the proof. Fix $\fa$ such that $G_{\fa}=S$ and set $p=|S|$. By Remark \ref{otherInt} we have  $\dim_k \HH{i}{\fm}{R/J_n}_{\fa} =  \dim_k \HH{i-p}{\fm_{\fa+}}{R_{\fa+}/J_{n,\,{\fa+}}}_{\fa^+}$ and by Lemma \ref{regGD} this dimension vanishes for  $|\fa^+|>\beta n$.  Therefore, 
\begin{align*}
\sum_{-\alpha n \ls |\fa| \ls\beta n, G_\fa=S} \dim_k \HH{i}{\fm}{R/J_n}_{\fa}&< ((\alpha +\beta)n)^{p}\lambda( \HH{i-p}{\fm_{\fa+}}{R_{\fa+}/J_{n,\,{\fa+}}}_{\succeq \f0})\\ 
&=O(n^{p})\cdot O(n^{d-p})=O(n^d),
\end{align*}
where the first equality  follows from Lemma \ref{geq0} as $\reg J_{n,\,{\fa+}}\ls \beta n$ by Lemma \ref{regGD}.
\end{proof}

If $I$ is a monomial ideal and $\lambda(\HH{i}{\fm}{R/I})<\infty$, then $\HH{i}{\fm}{R/I}_j=0$ for every $j<0$ (see \cite[Proposition 1]{Tak}), hence we recover the special case of Corollary \ref{corofMainMon} in which the ideals $(I_n)_{n\gs 0}$ are the powers of a monomial ideal.


\vspace{2mm}

Let $R=\CC[x_{i,j}]$  be the polynomial ring over the complex numbers in the $mn$ variables $\{x_{i,j} \mid 1\ls i\ls p, 1\ls j\ls q\}$. We view $R$ as the ring of polynomial functions on the matrices of size $m\times n$. We say that an $R$-ideal is  {\it GL-invariant} if it is invariant under the natural action by the group $GL=GL_p(\CC)\times GL_q(\CC)$. This class of ideals has been completely characterized in \cite{DEP} (see also \cite[Section 2.1]{Claudiu}). A natural example of $GL$-invariant ideals are powers {\it determinantal ideals}, i.e., the  ones generated by the $t\times t$ minors of a generic $p\times q$ matrix for any $1\ls t\ls \min\{p,q\}$. In \cite[6.1]{Claudiu}, Claudiu Raicu proves that if $I$ is a $GL$-invariant ideal that is the thickening of a determinantal ideal, then $\HH{i}{\fm}{R/I^n}_j=0$ for every  $i\ls  p+q-2$ and  $j<0$. 

\begin{cor}\label{determinantal}
Let $R=\CC[x_{i,j}]$  be the polynomial ring in $pq$ variables and $I$ a $GL$-invariant ideal that is the thickening of a determinantal ideal, then for every  $i\ls  p+q-2$ we have
$$\limsup_{n\rightarrow \infty}\frac{\lambda(\HH{i}{\fm}{R/I^n}) }{n^d}<\infty.$$ 
\end{cor}

\vspace{2mm}

For a homogeneous ideal $I$, we say $X=\Proj R/I$ is {\it locally complete intersection} (or simply {\it lci}), if $R_\fp/I_\fp$ is a complete intersection for any prime ideal $\fp\in \Proj R/I$. 

The next two results are explained to us by Robert Lazarsfeld, whom we thank. 

\begin{lemma}[Robert Lazarsfeld]\label{LemmaRob}
Let  $X\subset \PP^r$ be a lci scheme over a field of characteristic zero. Let $E$ be a vector bundle on $X$ and $\F$ a coherent sheaf on $X$, then there exists $\alpha\in \NN$ such that  
$$H^i(X, \Sym^n(E) \otimes \F \otimes \O_X(t))=0$$
for every $i>0$, $t> \alpha n$, and $n\gs 0$.
\end{lemma}
\begin{proof}
We can proceed as in the proof of Hartshorne's Theorem \cite[6.1.10]{PAG} via Fujita's Vanishing Theorem, to prove that if $E'$ is an ample vector bundle, then there exists $n_0:=n(E',\F)$ such that 
\begin{equation}\label{genHT}
H^i(X, \Sym^n(E') \otimes \F \otimes P)=0
\end{equation}
for every $i>0$, $n\gs n_0$, and all nef vector bundles $P$ on $X$.

Now, let $a_0$ be such that $E(a_0)$ is ample, then the result follows from \eqref{genHT} for $t\gs a_0 n$ and $n\gg 0$. Finally we may consider $a_0\gg 0$ and by Serre's Vanishing Theorem the statement follows for the remaining values of $n$.
\end{proof}

\begin{prop}[Robert Lazarsfeld]\label{Rob}
Let  $X\subset \PP^r$ be a lci scheme over a field of characteristic zero and let $I\subseteq \O_{\PP^r}$ be its ideal sheaf. Then there exists $\alpha\in \NN$ such that 
$$H^i(\PP^r, I^n(t))=0$$
for every $t<-\alpha n$, all $n\gs 1$, and $0\ls i\ls \dim X$.

\end{prop}

\begin{proof}

Set  $X_n:=\Proj \O_{\PP^r}/I^n$. We will proceed by induction on $n$ to prove that there exists $\alpha \in \NN$ such that $H^i(\O_{X_n}(t))=0$ for every $i<\dim X$ and $t< -\alpha n$. Let $\gamma \in \NN$ be such that $H^i(\O_{X_1}(t))=0$ for $t<-\gamma$ and  $N:=N_{X/\PP^r}$ the normal bundle of $X$ in $\PP^r$.  
Now, we apply Lemma \ref{LemmaRob} with $\F$ the canonical line bundle of $X$. From Serre's Duality Theorem and the isomorphism $(\Sym^n(N))^*=\Sym^n(N^*)$ it follows that there exists $\beta\in \NN$ such that 
$$H^i(X, \Sym^n(N^*)(t))=0 $$
for every  $i<\dim X$, $t<-\beta n$, and $n\gs 0$.

Consider the exact sequence 
$$0\rightarrow \Sym^n(N)^*\rightarrow \O_{X_{n+1}}\rightarrow \O_{X_n}\rightarrow 0$$
then the claim follows by induction with $\alpha=\max\{\beta, \gamma\}$. 

Finally, the result follows by the vanishing of negative twists of $\O_{\PP^r}$ and the following exact sequence 
$$0\rightarrow I^n\rightarrow \O_{\PP^r}\rightarrow \O_{X_n}\rightarrow 0.$$
\end{proof}

From Proposition \ref{Rob} and Theorem \ref{gradedFiniteLimpup} we obtain the following corollary.

\begin{cor}\label{lcilimsup}
Assume $k$ is of characteristic 0.  If $I$ is a homogeneous ideal such that $X=\Proj R/I$ is lci, then for every $i<\dim R/I$ we have
$$\limsup_{n\rightarrow \infty}\frac{\lambda(\HH{i}{\fm}{R/I^n}) }{n^d}<\infty.$$ 
\end{cor}

\vspace{.5mm}
\begin{remark}
Let $\Sing X$ be the singular locus of $X$. In \cite[3.1]{BBLSZ} the authors proved that if $i\ls \codim(\Sing X)$ then $\alpha=0$ satisfies the conclusion of Proposition \ref{Rob}.
\end{remark}

\vspace{.5mm}

\begin{cor}
Assume $k$ is of characteristic 0. If $I$ is a homogeneous radical ideal such that $\lambda(\HH{i}{\fm}{R/I^n}))<\infty$ for every $i<\dim R/I$ and $n\gg 0$,  then for every $i<\dim R/I$ we have
$$\limsup_{n\rightarrow \infty}\frac{\lambda(\HH{i}{\fm}{R/I^n}) }{n^d}<\infty.$$ 
\end{cor}
\begin{proof}
From local duality it follows that $R_{\fp}/I^nR_{\fp}$ is Cohen-Macaulay for every $n\gg 0$ and $\fp\in \Spec(R)\setminus \{\fm\}$. Therefore, by \cite{CN} $\Proj R/I$ is lci. The conclusion follows by Corollary \ref{lcilimsup}.
\end{proof}

\vspace{.5mm}

\begin{remark}
We note that if is not $I$ radical, the condition $\lambda(\HH{i}{\fm}{R/I^n}))<\infty$ for every $i<\dim R/I$ and $n\gg 0$ does not imply $\Proj R/I$ is lci. Indeed, if $I$ is any homogeneous ideal such that $\Proj R/I$ is lci, then $\Proj R/I^2$ need not be lci.
\end{remark}

\vspace{3mm}

\section{Positive lim inf}\label{linf}

In this section, we give classes of ideals for which the liminf of $\{\frac{\lambda(\HH{i}{\fm}{R/I^n})}{n^d}\}_{n\gs 0}$ is positive. The theme is that $R/I$ has sufficiently good singularities.

If $(R,\fm)$ is a Noetherian local ring of characteristic $p>0$. We say $R/I$ is {\it $F$-full} if $\HH{i}{\fm}{R/J}\rightarrow \HH{i}{\fm}{R/I}$ is surjective for every $J\subseteq I\subseteq \sqrt{J}$, see \cite[2.1]{Linquan1}.

Let $\F:R\rightarrow R$, given by $\F(x)=x^p$, be the {\it Frobenius endomorphism}.  The ring $R$ is said to be {\it $F$-pure}  if $\F\otimes M: R\otimes M \rightarrow R\otimes M$ is injective for every $R$-module $M$. It can be shown that every $F$-pure ring is $F$-full (see \cite[2.4]{Linquan1}).

\begin{thm}\label{FpureLiminf}
Let $(R,\fm)$ be a regular local ring of characteristic $p>0$ and dimension $d$. If $I$ is an $R$-ideal such that $R/I$ is $F$-full $($e.g. $F$-pure$)$, then for every $i$ such that $\HH{i}{\fm}{R/I}\neq 0$ we have $$\liminf_{n\rightarrow \infty } \frac{\lambda(\HH{i}{\fm}{R/I^n})}{n^d}>0.$$
\end{thm}
\begin{proof}
We show first that $\lim_{n\rightarrow \infty }\frac {\lambda(\HH{i}{\fm}{R/I^{[p^{n}]}})}{p^{n d}}$ exists and is positive. For this, let $^nR$ denote the $R$-algebra obtained by iterating $n$ times the Frobenius endomorphism. 
Since $R$ is regular $^nR$ is flat for every $n$ (\cite{Kunz}) and hence $$\HH{i}{\fm}{R/I^{[p^{n}]}}\cong\HH{i}{\fm}{R/I \otimes_R\, ^nR}\cong \HH{i}{\fm}{R/I}\otimes_R\, ^nR.$$ 
Therefore, either  $\lambda(\HH{i}{\fm}{R/I^{[p^{n}]}})=\infty$ for every $n$, or $\lim_{n\rightarrow \infty }\frac {\lambda(\HH{i}{\fm}{R/I^{[p^{n}]}})}{p^{n d}}$ 
is the Hilbert-Kunz multiplicity of the finite length $R$-module $\HH{i}{\fm}{R/I}$, which is positive by \cite[1.1]{DS}. 

For every $n\gs 1$ let $e_n=\lfloor \log_p (\frac{n}{\mu(I)})\rfloor. $ Clearly $I^n\subseteq I^{[p^{e_n}]}$ for every $n$. Since $R/I^{[p^{e_n}]}$ is  also  $F$-full (see \cite{DDM}), we have $\HH{i}{\fm}{R/I^n}\rightarrow \HH{i}{\fm}{R/I^{[p^{e_n}]}}$ is surjective for every $n \gs 1$. Therefore, $\lambda(\HH{i}{\fm}{R/I^n})\gs \lambda(\HH{i}{\fm}{R/I^{[p^{e_n}]}}).$
Now, 
$$\frac{p^{e_nd}}{n^d}=\left(\frac{p^{e_n}}{n}\right)^d > \left(\frac{\frac{n}{\mu(I)p}}{n}\right)^d= \left(\frac{1}{\mu(I)p}\right)^d.$$							   
Hence,
\begin{align*}
\liminf_{n\rightarrow \infty }\frac {\lambda(\HH{i}{\fm}{R/I^n})}{n^d}&\gs \liminf_{n\rightarrow \infty }\frac {\lambda(\HH{i}{\fm}{R/I^{[p^{e_n}]}})}{n^d}\\
					   &= \liminf_{n\rightarrow \infty } \frac {\lambda(\HH{i}{\fm}{R/I^{[p^{e_n}]}})}{p^{e_n d}}\cdot \frac{p^{e_n d}}{n^d}\\
					   &\gs \lim_{n\rightarrow \infty }\frac {\lambda(\HH{i}{\fm}{R/I^{[p^{n}]}})}{p^{n d}} \cdot   \liminf_{n\rightarrow \infty }  \frac{p^{e_n d}}{n^d}>0.
\end{align*}
\end{proof}

\begin{cor}\label{disconLemma}
Let $I$ be a square-free monomial ideal in $R= k[x_1,...,x_d]$ where $k$ is a field of positive characteristic.  Assume that $\tilde H_{i-1}(\Delta(I))\neq 0$. Then 
$$\liminf_{n\rightarrow \infty}\frac{\lambda(\HH{i}{\fm}{R/I^n})}{n^d}>0.$$
\end{cor}

\begin{proof}
This follows from the fact that when $k$ has positive characteristic and $I$ is square-free monomial ideal then $R/I$ is $F$-pure. Also, from Hochster's formula we have that $\HH{i}{\fm}{R/I}\neq 0$.
\end{proof}

\begin{cor}\label{DisconnectedH}
Let $I$ be a square-free monomial ideal in $R= k[x_1,...,x_d]$ where $k$ is a field (in any characteristic). If $\Delta(I)$ is disconnected, then 
$$\liminf_{n\rightarrow \infty}\frac{\lambda(\HH{1}{\fm}{R/I^n})}{n^d}>0.$$
\end{cor}

\begin{proof}
Theorem \ref{Takayama} shows that the Hilbert function of $\HH{1}{\fm}{R/I^n}$ is completely determined by the connected components of simplicial complexes $\Delta_\fa(I)$, which are independent of the characteristic of  the field $k$. Therefore we may assume $k$ is of positive characteristic and apply \ref{disconLemma}.

\end{proof}

\vspace{2mm}

The following is the second theorem of this subsection.

\begin{thm}\label{SmoothLiminf}
Let $R=k[x_1,\ldots,x_d]$ be a polynomial ring such that  $k$ is of characteristic 0. If $I$ is a homogeneous prime ideal such that $X=\Proj R/I$ is lci, then for every  $ i< \codim(\Sing X )$ 
such that $\HH{d-i}{I}{R}\neq 0$ we have $$\liminf_{n\rightarrow \infty } \frac{\lambda(\HH{i}{\fm}{R/I^n})}{n^d}>0.$$
\end{thm}

\begin{proof}
If $i=0$ the statement follows by  \cite[4.4]{UV} and the well known fact that if $\HH{d}{I}{R}\neq 0$ then $\ell(I)=d$. Therefore we may assume $i>0$.

Set $h=\height I$ and fix $n\gs 1$. Hence,
\begin{align*}
\lambda(\HH{i}{\fm}{R/I^n}))&\gs\sum_{j=0}^{n-d+h-2}\dim_k \HH{i}{\fm}{R/I^n}_j\\
&=\sum_{j=0}^{n-d+h-2}\dim_k \Ext_R^{d-i}(R/I^n,R)_{-d-j}\\
&=\sum_{j=-n-h+2}^{-d}\dim_k \Ext_R^{d-i}(R/I^n,R)_{j},
\end{align*}
where the first equality follows by graded local duality and the fact that $\omega_R\cong R(-d)$ (cf. \cite[Section 3.6]{BH}).

By \cite[2.18]{BBLSZ}, we have $\dim_k \HH{i}{\fm}{R/I^n}_j$ is constant for $j\ls n-d+h-2$, then so is $\dim_k \Ext_R^{d-i}(R/I^n,R)_{j}$ for $j\gs -n-h+2$. Since $\HH{d-i}{I}{R}=\displaystyle\lim_{\longrightarrow}\Ext_R^{d-i}(R/I^n, R)$, we have $\dim_k \Ext_R^{d-i}(R/I^n,R)_{j}=\HH{d-i}{I}{R}_j$ for every $j\gs -n-h+2$.

The ideal $I$ is locally generated by a regular sequence, then $\Supp \HH{d-i}{I}{R}\subseteq \{\fm\}$. Therefore, $\HH{d-i}{I}{R}$ is injective and
$\HH{d-i}{I}{R}\cong [E_R(k)(d)]^{\oplus N}$ 
for a finite positive number $N$ (cf. \cite[3.6]{Ly}). Now, from the equality $ E_R(k)=k[x_1^{-1},\cdots, x_d^{-1}]$ it follows that 
$$\sum_{j=-n-h+2}^{-d}\dim_k \Ext_R^{d-i}(R/I^n,R)_{j}=
N\sum_{j=-n-h+2}^{-d}\dim_k E_R(k)_{j+d}=N\sum_{j=0}^{n+h-d-2}\dim_k R_j.$$
Hence,
$$\liminf_{n\rightarrow \infty } \frac{\lambda(\HH{i}{\fm}{R/I^n})}{n^d}\gs 
N\liminf_{n\rightarrow \infty } \frac{\sum_{j=0}^{n+h-d-2}\dim_k R_j}{n^d}=Ne(R)=N>0,$$
where $e(R)$ denotes the {\it Hilbert-Samuel multiplicity} of  $R$.
\end{proof}

\vspace{3mm}

\section{Related works and open questions}\label{openq}

In this section we discuss some closely related literature and include various concrete questions left open by our work.

\subsection{Related works}

We have already mentioned the works \cite{BBLSZ,CutAdv,Claudiu} in the Introduction. We note that the main result of \cite{CutAdv} establishes the limit when $i=0$ even for graded families of $\fm$-primary ideals. Our Theorem \ref{mainMono} is inspired by this fact. 

Let $R=k[x_1,\ldots, x_d]$ and $\fm=(x_1,\ldots, x_d)$. In this graded case, bounding the total length of local cohomology modules are ultimately linked to bounds on the highest and lowest degrees of those modules. The former is closely related to the Castelnuovo-Mumford regularity of powers of $I$, a subject that has drawn intense interests in recent years. We refer to the survey \cite{Cut} for a comprehensive list of references. 

When $i=0$ and $I$ is a monomial ideal one can interpret the limit $\epsilon(I)$ as volume of some polytopes as in \cite{JM}. One can also compute $\epsilon(I)$ for determinantal ideals using some integrals that arise in random matrix theory, see  \cite{JMV}.

When $R$ has positive characteristic $p$, it is natural to study local cohomology of higher Frobenius powers of  $I$, $(I^{[p^e]})_{e\gs 0}$. This setting has recently been analyzed in detail in \cite{DS}. The connections of these limits to tight closure, invariants of vector bundles, and volumes of certain convex bodies have been explored in \cite{Bre, BreCa, DS, DW, JH}.

\subsection{Open problems}
Perhaps the most fundamental open problem is:

\begin{Question}
If $I$ is a homogeneous ideal such that $\lambda(\HH{i}{\fm}{R/I^n})<\infty$ for $n\gg 0$. Is
$\limsup_{n\rightarrow \infty}\frac{\lambda(\HH{i}{\fm}{R/I^n})}{n^d}$ finite? 
\end{Question}

In light of Theorem \ref{gradedFiniteLimpup}, the previous question reduces to the following.

\begin{question}
Let $I$ be homogeneous ideal and assume that $\lambda(\HH{i}{\fm}{R/I^n})<\infty$ for $n\gg 0$. Does there exist $\alpha\in \ZZ$ such that $\HH{i}{\fm}{R/I^n}_j=0$ for every $n\gg 0$ and  $j < -\alpha n$? 
\end{question}

In view of \cite{BBLSZ} and \cite{Claudiu}, this question can be seen as an ``asymptotic Kodaira vanishing for thickenings of $I$".  
\begin{question}
When is $\lim_{n\rightarrow \infty}\frac{\lambda(\HH{i}{\fm}{R/I^n})}{n^d}$ rational? 
\end{question}

Note that even when $i=0$ and $I$ defines a smooth projective surface, the limit can be irrational (see \cite{CHST}). However, it is reasonable to hope that it will be rational for ideals with highly combinatorial structure, for example those that are $GL$-invariant as in \cite{Claudiu} or those that define toric varieties. 

Finally, note that the limit we study has a strong interpretation as volume of certain objects associated to a monomial ideal $I$ as in Section \ref{smono2}. It is reasonable to ask whether one can extend this understanding to more general situations. 

\begin{Question}
Let $I$ be a homogeneous ideal.  Can we interpret   $\lim_{n\rightarrow \infty}\frac{\lambda(\HH{i}{\fm}{R/I^n})}{n^d}$ using volumes of certain geometric bodies associated to $I$? 
\end{Question}





\begin{thebibliography}{99}
\addcontentsline{toc}{section}{Bibliography}

\bibitem{BS} D. Bayer and M. Stillman, \emph{A criterion for detecting m-regularity},  Invent. Math. {\bf 87} (1987), 1--11.


\bibitem{BBLSZ} B. Bhatt, M. Blickle, G. Lyubeznik, A. K. Singh, and W. Zhang, \emph{Stabilization of the cohomology of thickenings}, preprint, arXiv:1605.09492 (2016).

\bibitem{Bre} H.~Brenner, \emph{Irrational Hilbert-Kunz multiplicities}, preprint, arXiv:1305.5873 (2013).

\bibitem{BreCa} H.~Brenner and A.~Caminata,  \emph{Generalized Hilbert-Kunz function in graded dimension two}, arXiv:1508.05771 (2015), to appear in Nagoya Math. J.


\bibitem{Br} M.~ Brodmann, \emph{The asymptotic nature of the analytic spread}, Math. Proc. Cambridge Philos. Soc. {\bf 86} (1979), 35--39.


\bibitem{BrSh} M. P. Brodmann and R. Y. Sharp, \emph{Local cohomology. An algebraic introduction with geometric applications}, Cambridge Studies in Advanced Mathematics, {\bf 136} Cambridge University Press, Cambridge, 2013.



\bibitem{BH} W.~ Bruns and J.~Herzog, \emph{Cohen-Macaulay Rings}, Cambridge Studies in Advanced Mathematics, {\bf 39}, Cambridge, Cambridge Univ. Press, 1993.


\bibitem{Normaliz} W.~Bruns, B.~Ichim, T.~Römer, R.~Sieg, and C.~Söger: \emph {Normaliz. Algorithms for rational cones and affine monoids}. Available at https://www.normaliz.uni-osnabrueck.de.

\bibitem{CN} C.~Cowsik and M.~Nori, \emph{On the fibres of blowing up}, J. Indian Math. Soc. {\bf 40} (1976), 217--222.

\bibitem{CutAdv} S. D.~Cutkosky, \emph{Asymptotic multiplicities of graded families of ideals and linear series}, Adv. Math. {\bf 264} (2014), 55--113.

\bibitem{Cut} S. D.~Cutkosky, \emph{Limits in commutative algebra and algebraic geometry}, Commutative Algebra and Noncommutative Algebraic Geometry, MSRI Publications, {\bf 67}, Cambridge Univ. Press, New York, 2015, 141--162.

\bibitem{CHT} S.D.~Cutkosky, J.~Herzog, and N.V.~Trung, \emph{Asymptotic behaviour of Castelnuovo–Mumford regularity}, Compos. Math. {\bf 80} (1999), 273--297.

\bibitem{CHST} S. D.~Cutkosky, H. T.~H\`a, H.~Srinivasan, and E.~Theodorescu, \emph{Asymptotic behavior of the length of local cohomology}, Canad. J. Math. {\bf 57} (2005), 1178--1192.

\bibitem{DDM} H.~Dao, A.~De Stefani, and L.~Ma, \emph{Cohomologically full rings}, {in progress}.


\bibitem{DS} H.~Dao and I.~Smirnov, \emph{On generalized Hilbert-Kunz function and multiplicity}, preprint, arxiv:1305.1833 (2013).

\bibitem{DW} H.~Dao and K.~Watanabe, \emph{Some computations of the generalized Hilbert-Kunz function and multiplicity}, Proc. Amer. Math. Soc. {\bf 144} (2016), 3199--3206.



\bibitem{DEP} C. De Concini, D. Eisenbud, and C. Procesi, \emph{Young diagrams and determinantal varieties}, Invent. Math. {\bf 56} (1980),  129--165.



\bibitem{GS} D.~Grayson and M.~Stillman, \emph{Macaulay2, a software system for research in algebraic geometry}, Available at http://www.math.uiuc.edu/Macaulay2.

\bibitem{JH} D. Hernandez and J. Jeffries, \emph{Local Okounkov bodies and limits in prime characteristic}, preprint, arXiv:1701.02575 (2017).



\bibitem{HH} J. Herzog and T. Hibi, \emph{Monomial ideals}, Graduate Texts in Mathematics, {\bf 260}, Springer-Verlag London, Ltd., London, 2011.


\bibitem{HHT} J.~Herzog, T.~Hibi, and N.V.~Trung, \emph{Symbolic powers of monomial ideals and vertex cover algebras}, Adv. Math. {\bf 210} (2007), 304--322.

\bibitem{HPV} J. Herzog, T. J. Puthenpurakal, and J. K. Verma, \emph{Hilbert polynomials and powers of ideals}, Math. Proc. Cambridge Philos. Soc. {\bf 145} (2008), 623--642.





\bibitem{HS} C.~Huneke and I.~Swanson, \emph{Integral closures of ideals, rings and modules}, London Math. Society Lecture Note Series {\bf 336}, Cambridge University Press, 2006.

\bibitem{JM} J. Jeffries and J. Monta\~no, \emph{The j-multiplicity of monomial Ideals}, Math. Res. Lett. {\bf 20} (2013), 729--744.

\bibitem{JMV} J. Jeffries, J. Monta\~no, and M. Varbaro \emph{Multiplicities of classical varieties}, Proc. London Math. Soc. {\bf 110} (2015), 1033--1055. 


\bibitem{KV} D. Katz and J. Validashti, \emph{Multiplicity and Rees valuations}, Collect. Math. {\bf 61} (2010), 1--24.

\bibitem{Kod} V.~Kodiyalam, \emph{Asymptotic behaviour of Castelnuovo-Mumford regularity}, Proc. Amer. Math. Soc. {\bf 128} (2000), 407--411.

\bibitem{Kunz} E.~Kunz, \emph{Characterizations of regular local rings for characteristic $p$}, Amer. J. Math. {\bf 91} (1969), 772--784. 

\bibitem{LT}  H.M. Lam and N.V. Trung, \emph{Associated primes of powers of edge ideals and ear decompositions of graphs}, preprint, arXiv:1506.01483 (2015).

\bibitem{Ly} G. Lyubeznik, \emph{Finiteness properties of local cohomology modules (an application of D- modules to commutative algebra)}, Invent. Math. {\bf 113} (1993), 41--55.

\bibitem{Linquan1} L. Ma and P. H. Quy, \emph{Frobenius actions on local cohomology modules and deformation}, arXiv:1606.02059 (2016), to appear in Nagoya Math. J.

\bibitem{Mc} P.~McMullen, \emph{Lattice invariant valuations on rational polytopes}, Arch. Math. (Basel), {\bf 31} (1978), 509--516.

\bibitem{MT} N.C. Minha and N.V. Trung, \emph{
Cohen–Macaulayness of powers of two-dimensional squarefree monomial ideals}, J. Algebra {\bf 322} (2009), 4219--4227.

\bibitem{PAG} R. Lazarsfeld, \emph{Positivity in algebraic geometry I and II},  Ergebnisse der Mathematik und ihrer Grenzgebiete. 3. Folge, vol. 48 and 49. Springer-Verlag, Berlin, 2004. 

\bibitem{Claudiu}  C.~Raicu, \emph{Regularity and cohomology of determinantal thickenings}, preprint, arXiv:1611.00415 (2016).


\bibitem{Sba} E.~Sbarra, \emph{Upper bounds for local cohomology for rings with given Hilbert function}, Comm. Algebra {\bf 29} (2001), 5383–-5409.


\bibitem{Schwede}  K. Schwede, \emph{F-injective singularities are Du Bois}, Amer. J. Math. {\bf 131} (2009), 445–473.


\bibitem{SVV} A. Simis, W. Vasconcelos, R. Villarreal, \emph{On the ideal theory of graphs}, J. Algebra {\bf 167} (1994), 389--416.

\bibitem{St} R.P.~Stanley, \emph{Decompositions of rational convex polytopes}, Ann. Discrete Math. {\bf 6} (1980), 333--342.


\bibitem{SW} A. Singh and U. Walther, \emph{Local cohomology and pure morphisms}, Illinois J. Math., Special Issue in Honor of Phil Griffith, {\bf 51} (2007), 287--298.

\bibitem{Tak} Y. Takayama, \emph{Combinatorial characterizations of generalized Cohen-Macaulay monomial ideals}, Bull. Math. Soc. Sci. Math. Roumanie (N.S.), {\bf 48} (2005), 327–-344.

 \bibitem{UV} B. Ulrich and J. Validashti, \emph{Numerical criteria for integral dependence}, Math. Proc. Cam- bridge Philos. Soc., {\bf 151} (2011), 95--102.

\bibitem{Vas} W.~Vasconcelos, \emph{Integral closure: Rees algebras, multiplicities, algorithms}, Springer Monographs in Mathematics. Springer, New York, 2005.


\bibitem{Woods} K. Woods, \emph{Presburger arithmetic, rational generating functions, and quasi-polynomials}, J. Symb. Log. {\bf 80} (2015), 433--449.

\end{thebibliography}
\end{document}